\documentclass[reqno]{amsart}
\usepackage{amsfonts, amsmath, amsthm, amssymb, latexsym, graphicx}

\theoremstyle{plain}
\newtheorem{theorem}{Theorem}[section]
\newtheorem{corollary}[theorem]{Corollary}
\newtheorem{lemma}[theorem]{Lemma}
\newtheorem{proposition}[theorem]{Proposition}

\theoremstyle{definition}
\newtheorem{definition}[theorem]{Definition}
\theoremstyle{remark}
\newtheorem{remark}[theorem]{Remark}
\newtheorem{remarks}[theorem]{Remarks}

\numberwithin{equation}{section}

\title[Sonine formulas and  intertwining operators in Dunkl theory]{Sonine formulas and  intertwining operators  in Dunkl theory}
\author{Margit R\"osler}
\address{Insitut f\"ur Mathematik, Universit\"at Paderborn, Warburger Str. 100, D-33098 Paderborn, Germany}
\email{roesler@math.upb.de}
\author{Michael Voit}
\address{Fakult\"at Mathematik, Technische Universit\"at Dortmund,
          Vogelpothsweg 87,
          D-44221 Dortmund, Germany}
\email{michael.voit@math.tu-dortmund.de}


\begin{document}
\date{\today}

\begin{abstract} 
Let $V_k$ denote Dunkl's intertwining operator associated with some root system $R$ and multiplicity  $k$. For two multiplicities $k, k^\prime$ on $R$, we study the intertwiner $V_{k^\prime,k} = V_{k^\prime}\circ V_k^{-1}$  between Dunkl operators with multiplicities $k$ and $k^\prime.$ It has been a long-standing  conjecture that  $V_{k^\prime,k}$ is  positive if $k^\prime \geq k \geq 0.$ We disprove this conjecture by constructing counterexamples for root system $B_n$. This matter is closely related to the existence of Sonine-type integral representations between Dunkl kernels and Bessel functions with different multiplicities. In our examples, such Sonine formulas do not exist. As a consequence, we obtain necessary conditions on Sonine formulas for Heckman-Opdam hypergeometric functions of type $BC_n$ and conditions for positive branching coefficients between  multivariable Jacobi polynomials.
\end{abstract}

\maketitle

\section{Introduction}

In the theory of rational Dunkl operators initiated by C.F. Dunkl in \cite{D1, D2},
 the intertwining operator plays a significant role. This operator intertwines Dunkl operators with the usual partial derivatives on some Euclidan space. To become more precise, let $R$ be a (not necessarily crystallographic) root system in a 
 finite-dimensional Euclidean space 
 $(\frak a, \langle\,.\,,\,.\,\rangle)$ with  finite Coxeter group $W,$  and fix a   $W$-invariant function $k:R \to \mathbb C$ (called multiplicity function) with 
 $\text{Re}\, k \geq 0$.  Denote by $\{T_\xi(k), \xi\in \frak a\}$ the associated commuting family of rational Dunkl operators. The 
 intertwining operator $V_k$ is then characterized as the unique isomorphism on the vector space $\mathcal P=\mathbb C[a]$  
  of polynomial functions on $\frak a$ which preserves the degree of homogeneity and satisfies 
$$ V_k(1)=1, \quad T_\xi(k) V_k = V_k \partial_\xi \quad \text{ for all } 
  \xi \in \frak a;$$ c.f. \cite{DJO}. The Dunkl kernel $E_k$ associated with $R$ and $k$, which solves the joint eigenvalue problem for the $T_\xi(k)$ 
  and generalizes the usual exponential kernel,  can be represented by means of the intertwiner $V_k$ as
  $ E_k(x,z) = V_k\bigl(e^{\langle \,.\,,  z\rangle}\bigr)(x)$ for all $x\in \frak a $ and $z\in \frak a_{\mathbb C},$
  where $\frak a_{\mathbb C}$ denotes the complexification of $\frak a.$ 

 For nonnegative multiplicities 
 $k\geq 0$, it was shown in \cite{R1} that $V_k$ is  positive on $\mathcal P,$ i.e. for $p\in \mathcal P$ with
 $p\geq 0$ on $\frak a$, it follows that $V_kp\geq 0$ on $\frak a.$ Further, for each $x\in \frak a$ there exists a unique  probability measure 
$\mu_{x}^k$ on $\frak a$ such that 
\begin{equation}\label{Dunkl_kernel_int} E_k(x,z) = \int_{\frak a} e^{\langle \xi,z\rangle } d\mu_x^k(\xi), \quad \forall \, 
x \in \frak a, z\in \frak a_{\mathbb C}.
\end{equation}
The representing measure $\mu_x^k$ is compactly supported 
with $\text{supp}\,\mu_x^k \subseteq\text{co}(W.x)$, the convex hull of 
the $W$-orbit of $x$. 
Formula \eqref{Dunkl_kernel_int}  generalizes the Harish-Chandra integral representation for the spherical functions  of a symmetric space of Euclidean type.
Indeed, for 
certain half-integer valued multiplicities $k$, the Bessel functions 
$$J_k(x,z) = \frac{1}{|W|}\sum_{w\in W} E_k(wx,z), \quad z\in \frak a_{\mathbb C},$$
can be interpreted as the spherical functions of a Cartan motion group, where $R$ and $k$ are determined by the root space data of the 
underlying symmetric space, see \cite{O, dJ2} for details. 
In these geometric cases, the integral formula for $J_k$ obtained from
\eqref{Dunkl_kernel_int} by taking $W$-means is a direct consequence of the Harish-Chandra formula  together with 
Kostant's convexity theorem \cite[Propos. IV.4.8 and Theorem IV.10.2]{Hel}.

\medskip

In this paper, we shall consider  two multiplicities $k, k^\prime$ on $R$ 
with  $k^\prime \geq k\geq 0$ (i.e., 
$k^\prime(\alpha) \geq k(\alpha)\geq 0 \, \forall \alpha \in R$) and study the operator 
$$ V_{k^\prime\!,k} := V_{k^\prime} \circ V_k^{-1} .$$
Notice that $V_{k^\prime,0} = V_{k^\prime}.$ The operator $V_{k^\prime\!,k}$ intertwines the 
Dunkl operators  with multiplicities $k$ and $k^\prime$, 
$$ T_\xi(k^\prime) V_{k^\prime\!,k} = V_{k^\prime\!,k}\, T_\xi(k) \quad \text{ for all } \xi \in \frak a.$$
It has been a long-standing conjecture that $V_{k^\prime\!, k}$ is also positive on polynomials, which is (as will be explained in Section \ref{intertwining}) equivalent to the statement that for each $x\in \frak a$, there exists
a compactly supported probability measure $\mu_x^{k^\prime\!,k}$ on $\frak a$ such that 
\begin{equation}\label{Sonine_E}  E_{k^\prime}(x,z) = \int_{\frak a} E_k(\xi,z)\, d\mu_x^{k^\prime\!,k}(\xi)
\quad \text{for all } z\in \frak a_{\mathbb C}.
\end{equation}
Note that $\eqref{Sonine_E}$ implies an analogous formula for the Bessel function:
\begin{equation}\label{Sonine_B}  J_{k^\prime}(x,z) = \int_{\frak a} J_k(\xi,z) \,d\widetilde \mu_x^{k^\prime\!,k}(\xi) \quad (z\in \frak a_{\mathbb C})
 \end{equation}
with some $W$-invariant probability measures $\widetilde \mu_x^{k^\prime\!,k}.$

\medskip

In the rank-one case with 
$R= \{\pm 1\}\subset \mathbb R$, one has
$\, J_k(x,y) = j_{k-1/2}(ixy) \,$
with the (modified) one-variable Bessel function
\begin{equation}\label{B_eindim} j_\alpha(z)  = \, _0F_1(\alpha+1;-z^2/4) \quad\quad 
(\alpha\in\mathbb C\setminus\{-1,-2,\ldots\}).\end{equation}
In this case, formula \eqref{Sonine_B} is just the classical 
Sonine formula  (\cite[formula (3.4)] {A2}):
\begin{equation}\label{intrep-1-dim}
j_{\alpha+\beta}(z)= 2\frac{\Gamma(\alpha+\beta+1)}{\Gamma(\alpha+1)\Gamma(\beta)}
\int_{0}^1j_{\alpha}(zx)  x^{2\alpha+1}(1-x^2)^{\beta-1} dx
\end{equation}
for all $\alpha, \beta \in \mathbb R$ with $\alpha >-1$ and $\beta >0.$

In the rank-one case also the operator $V_{k^\prime\!, k}$ with $k^\prime > k \geq 0$  is known to be positive. Indeed, 
Y. Xu obtained in \cite{X} 
an explicit positive integral representation for $V_{k^\prime\!,k}$  which leads to a positive 
Sonine-type representation  for 
the rank-one 
Dunkl kernel, see Remark \ref{Sonine_Soltani} for details.

In the present paper, we shall construct examples which reveal that the above positivity conjecture  is not true in 
general. Our examples are related to root system 
$$ B_n = \{ \pm e_i, \, \pm e_i \pm e_j, \, 1\leq i < j \leq n\} \subset \mathbb R^n$$
with $n\geq 2$, where  multiplicities are denoted as $k=(k_1,k_2),$ with $k_1$ and $k_2$ the values of $k$ on $e_i$ and $e_i\pm e_j$, respectively. 
We prove that for $k= (k_1, k_2)$ with $k_1\geq 0, k_2 >0$ and $k^\prime = k^\prime(h)= (k_1+h, k_2)$ with $\,h>-k_1$, the Bessel function 
$J_{k^\prime(h)}^B$ 
of type $B_n$ cannot have a positive Sonine representation with respect to $J_k^B$ if $h$ is not contained in the set
$$ \Sigma(k_2):= \{r\in \mathbb R: r > k_2(n-1)\} \,\cup \, \{ jk_2-m\,: j = 0, 1, \ldots n-1; m= 0, 1, 2\ldots \}.$$
This implies that for  $h \notin \Sigma(k_2),$ the intertwining operator 
$V_{k^\prime(h)\!,k}$ is not positive.
More generally, we shall consider also complex multiplicities and obtain similar conditions for Sonine representations with 
complex bounded Radon measures.

The proof of our main  result, which is contained  in Corollary \ref{main_Bessel}, is based on the fact
that the Bessel function of type $B_n$ can be expressed as a multivariable $\,_0F_1$-hypergeometric function 
in the sense of \cite{K} (see also \cite{BF1}).  Via Kadell's \cite{Ka} generalization of the Selberg integral one obtains an explicit Sonine formula for this 
hypergeometric function and 
therefore also for the Bessel function  $J_{k^\prime(h)}^B$ in terms of $J_k^B$ within the range $\text{Re}\, h > k_2(n-1).$ 
This explicit formula allows a distributional extension to a larger range of the parameter $h$, 
 which is based on results of \cite{dJ2} for the intertwiner $V_k$. 
Employing arguments 
of Sokal \cite{S} for the characterization of Riesz distributions on symmetric cones, we then obtain necessary conditions on $h$ under which
our distributional Sonine formulas can actually be given by positive or complex measures. 
Indeed, the set $\Sigma(k_2)$ is similar to the so-called Wallach set, which describes 
those Riesz distributions of a symmetric cone which are actually positive measures. Our counterexamples seem to be  specific for the $B_n$ case and concern only the multiplicities on the roots $\pm e_i$. We still conjecture that 
$V_{k^\prime, k}$ is positive for $k^\prime > k > 0$ in the $A_{n-1}$-case. 
We also mention that for Bessel functions on symmetric cones, Sonine formulas were recently studied in \cite{RV2}.

Our results on Sonine formulas in the rational Dunkl setting are contained  in Section \ref{B_n},  which is preceded by
preparations for intertwining operators in Section \ref{intertwining}. In Section \ref{trig},  we apply the results from Section \ref{B_n} to the trigonometric theory 
of Heckman, Opdam and Cherednik (see \cite{HS, O2}) which generalizes the spherical harmonic analysis on Riemannian symmetric spaces of the non-compact and compact type. 
 We shall use a well-known contraction procedure from the trigonometric to the rational case in order to derive 
necessary conditions on the existence of Sonine-type 
integral representations between 
hypergeometric functions and Heckman-Opdam polynomials (also called Jacobi polynomials) associated with root system $BC_n$ as well as 
the positivity of branching  coefficients  between two such polynomial systems with different multiplicities. These results are complemented by 
motivating examples in rank one and the case of symmetric spaces. Let us mention that in geometric cases, branching rules and Sonine-type formulas for Bessel functions were 
recently also studied in \cite{HZ} in connection with the geometry of moment mappings.

\section{Intertwining operators and Sonine formula for Dunkl kernels}\label{intertwining}

We start with some background and notation in rational Dunkl theory  supplementing the material in the introduction. For more information, 
the reader is referred to \cite{dJ, O, DJO, dJ2, DX}  and the references cited there. 
Again, $R$ is a root system in a finite-dimensional Euclidean space $(\frak a, \langle\,.\,,\,.\,\rangle)$  and  
$W=W(R)$ be
the associated finite Coxeter group.  We assume in this section that $R$ is reduced, but not necessarily crystallographic. Let 
$\,\mathcal K=\{k:R\to \mathbb C\,: \,  k \text{ is } W
\text{-invariant}\}$ 
denote the space of multiplicity functions on $R$.  
For two multiplicities $k, k^\prime\in \mathcal K$ we write $k^\prime \geq k$ ($\text{Re}\, k^\prime \geq \text{Re}\, k$) 
if $k^\prime(\alpha) \geq k(\alpha)$ 
($\text{Re}\, k^\prime (\alpha) \geq \text{Re}\,k(\alpha)$) for all $\alpha\in R.$ 
The Dunkl operators
 associated with $R$ and  $k\in \mathcal K$ are given by
 $$ T_\xi(k) = \partial_\xi  + \frac{1}{2}\sum_{\alpha\in R}
 k(\alpha) \langle \alpha, \xi\rangle \frac{1}{\langle \alpha,\,.\,\rangle}(1-\sigma_\alpha) , \quad \xi \in \frak a
 $$
 where the action of $W$ on functions $f: \frak a \to \mathbb C$  is given by $w.f(x) = f(w^{-1}x).$
 It was shown in \cite{D1} that the  $T_\xi(k), \, \xi \in \frak a$   commute. 
 A multiplicity $k$ is called regular if the joint kernel of the $T_\xi(k)$, considered as linear operators on 
 $\mathcal P=\mathbb C[\frak a],$ consists of the constants only.  This is equivalent to the existence of a (necessarily unique) intertwining operator 
 $V_k$ as described in the introduction. 
 The set $\mathcal K^{reg}$ of regular multiplicities is open in $\mathcal K$ and contains the set $\{k\in \mathcal K: \text{Re}\, k\geq 0\}$, see \cite{DJO}. 
 Moreover, for each $k\in \mathcal K^{reg}$ and $y \in \frak a_\mathbb C,$ there exists a unique solution $f = E_k(\,.\,,y)$ of the 
joint eigenvalue problem
$$ T_\xi(k)f = \langle \xi,y\rangle f\quad \forall \, \xi \in \frak a, \,\, f(0)=1.$$
The function $E_k$ is called the Dunkl kernel. The mapping $(k,x,y)\mapsto E_k(x,y)$ is analytic on $\mathcal K^{reg}\times \frak a_\mathbb C\times \frak a_\mathbb C$ 
and satisfies $ E_k(x,y) = E_k(y,x)$ as well as 
 $$   E_k(\lambda x,y) = E_k(x, \lambda y), \quad E_k(wx, wy) = E_k(x,y)\quad (\lambda \in \mathbb C, w \in W).$$ 
We shall from now on always assume that $\text{Re}\, k\geq 0$. In this case,  the following estimate for the Dunkl kernel is due to \cite{dJ}:
\begin{equation}\label{estimableitung} \vert  E_k(x,z)\vert \leq \sqrt{|W|}\,e^{\max_{w\in W}\langle wx, \text{Re}\, z\rangle} \quad \forall \, x\in \frak a, z \in \frak a_{\mathbb C}. \end{equation}

Denote by $\mathcal E(\frak a)$ the space $ C^\infty(\frak a)$ of smooth functions on $\frak a$, equipped with its usual 
Fr\'echet space topology. According to \cite{dJ2}, the operator $V_k$ (uniquely) extends to a homeomorphism of $\mathcal E(\frak a)$ retaining 
the intertwining property. Thus
$$ E_k(x,z) = V_k\bigl(e^{\langle \,.\, , z\rangle}\bigr)(x), \quad \forall \, x\in \frak a, z \in \frak a_{\mathbb C}.$$

We next recapitulate some facts from \cite{dJ}
about the Dunkl transform which was introduced in \cite{D3}.
Consider the (complex-valued) $W$-invariant weight
$$ \omega_k(x) = \prod_{\alpha \in R} |\langle\alpha, x\rangle|^{k(\alpha)}. $$
The Dunkl transform associated with $R$ and $k$  on $L^1(\frak a, |\omega_k|)$   is defined  by 
\[ \widehat f^{\,k}(\xi) = \int_{\frak a} f(x) E_k(x,-i\xi) \omega_k(\xi)d\xi, \quad \xi \in \frak a.\]
The Dunkl transform $\mathcal D_k: f\mapsto \widehat f^{\,k}$ is a homeomorphism of the Schwartz space $\mathcal S(\frak a)$ with inverse 
$$\mathcal D_k^{-1}f(x) = \frac{1}{c_k^2} \mathcal D_kf(-x), \quad  c_k = \int_{\frak a } e^{-|x|^2/2}\omega_k(x)dx.$$
Notice that $c_k\not=0$ by \cite[Cor. 4.17]{dJ}. Dunkl operators act continuously on $\mathcal S(\frak a)$ 
and therefore also on the space $\mathcal S^\prime(\frak a)$ 
of tempered distributions on $\frak a$, via
\begin{align}\label{Dunkl_dist} \langle T_\xi(k)u, \varphi\rangle := -\langle u, T_\xi(k) \varphi \rangle, \quad u \in\mathcal S^\prime (\frak a), \, \, \varphi \in \mathcal S(\frak a).\end{align}
Moreover, the Dunkl transform extends to a homeomorphism 
$ u \mapsto \widehat u^{\,k}$ of $\mathcal S^\prime(\frak a)$ by
$$ \langle \widehat u^{\,k}, \varphi\rangle:= \langle u, \widehat \varphi^{\, k}\rangle, \quad \varphi \in \mathcal S(\frak a).$$

For $R>0$ let 
$\,B_R(0):= \{x\in \frak a: |x|<R\}$ and $ \overline B_R(0):= \{x\in \frak a: |x|\leq R\},$ where $|\,.\,|$ denotes the norm associated with the given inner product. 
We shall use the following facts concerning the intertwiner $V_k$. 

\begin{proposition}\cite[Theorem 5.1]{dJ2}\label{deJeu}
\begin{enumerate}\itemsep=2pt
\item[\rm{(1)}] If $\varphi\in \mathcal E(\frak a)$ vanishes on $B_R(0),$ then also 
$V_k \varphi$ and $V_k^{-1}\!\varphi\, $ vanish on $B_R(0).$
\item[\rm{(2)}] Let $\varphi\in \mathcal S(\frak a).$ Then for all $x\in \frak a,$ 
\begin{enumerate}
\item[\rm{(a)}]
$\displaystyle V_k\varphi(x) =  \frac{c_k^2}{c_0^2} \mathcal D_k^{-1}(\omega_k^{-1} \mathcal D_0)\varphi(x)= 
\frac{1}{c_0^2}\! \int_{\frak a} \widehat \varphi^{\,0}(\xi) E_k(ix,\xi)\,d\xi.$
\item[\rm{(b)}] 
$ \displaystyle V_k^{-1} \!\varphi(x) = \frac{c_0^2}{c_k^2} \mathcal D_0^{-1}(\omega_k\mathcal D_k)\varphi(x) =
\,\frac{1}{c_k^2}\! \int_{\frak a} \widehat \varphi^{\,k}(\xi) e^{i\langle x,\xi\rangle}\,\omega_k(\xi)d\xi.$
\end{enumerate}
\end{enumerate}
\end{proposition}

For an open subset $\Omega\subseteq \frak a$ we denote by $\mathcal D(\Omega) = C_c^\infty(\Omega)$ the set of test functions
and by $\mathcal D^\prime(\Omega)$ the set of distributions on 
$\Omega$.  
Recall that the topological dual $\mathcal E^\prime(\Omega)$ of $\mathcal E(\Omega)$ coincides with the set of compactly supported distributions on
$\Omega,$ and that
compactly supported distributions on $\frak a$ are tempered. 

\begin{definition}
We define the Dunkl-Laplace transform of $u\in \mathcal E^\prime (\frak a)$ by 
$$ \mathcal L_ku: \frak a_{\mathbb C}  \to \mathbb C, \, \mathcal L_ku(z) := \langle u(x), E_k(x,-z)\rangle,$$
where the notion $u(x)$ indicates that $u$ acts on functions of the variable $x$. 
\end{definition}

As in the classical case, 
we have the following fact for compactly supported distributions, c.f. also \cite{BSO}.

\begin{lemma}\label{Laplace_allg}
Let $u\in \mathcal E^\prime(\frak a).$ Then $\mathcal L_k u$ is analytic on $\frak a_\mathbb C$, and the  Dunkl transform $\widehat u^{\,k} $ 
is a regular tempered distribution given by $\mathcal L_k u$ in the 
sense that
\begin{equation}\label{id_laplace}  \langle \widehat u^{\,k}, \varphi\rangle = \int_{\frak a } \varphi(\xi) \mathcal L_k u (i\xi) \,\omega_k(\xi)d\xi\,, \quad \varphi\in \mathcal S(\frak a).
\end{equation}

\end{lemma}

\begin{proof} This is the same as in \cite[Theorem 7.1.14]{H} for the classical case. We briefly note the steps:
According to \cite[Theorem 2.1.3]{H} $ \mathcal L_k u$ is smooth on $\frak a_{\mathbb C}$ and differentiations  with respect to $z$ may be taken in the argument $E_k(x, -z)$. As this kernel is analytic in $z$, the same follows for $ \mathcal L_k u(z).$
For the proof of \eqref{id_laplace}, it suffices to consider $\varphi \in \mathcal D(\frak a).$ 
By the Fubini theorem for compactly supported distributions
we obtain 
\begin{align} \langle u, \widehat \varphi^{\,k} \rangle & = 
 \big\langle u(x), \int_{\frak a} \varphi(\xi) E_k(-ix, \xi) w_k(\xi) d\xi\,\big\rangle  \notag \\
 & = 
 \big\langle u(x) \otimes \varphi(\xi)\omega_k(\xi), E_k(-ix, \xi)\big\rangle \, = \int_{\frak a} \varphi(\xi) \langle u(x), E_k(-ix, \xi)\rangle \, \omega_k(\xi) d\xi. 
 \notag 
 \end{align}
This implies the assertion.
\end{proof}


\begin{corollary}\label{laplace_injectivity} Let $\, u\in \mathcal E^\prime(\frak a ).$ 
\begin{enumerate}
\item[\rm{(1)}] If $\mathcal L_k u = 0$, then $u=0.$
\item[\rm{(2)}] Suppose that $m\in M_b(\frak a)$ is a complex bounded Radon measure satisfying
$$ \mathcal L_ku (i\xi) = \int_{\frak a} E_k(x,-i\xi) dm(x) \quad\text{for all } \xi \in \frak a.$$
Then $m=u.$
\end{enumerate}
\end{corollary}

\begin{proof} (1) is obvious by the above Lemma, because
the Dunkl transform is a homeomorphism of $\mathcal S^\prime(\frak a).$  

(2) Consider $m$ as a tempered distribution on $\frak a.$ By Lemma \ref{Laplace_allg} and our assumption  we obtain for test functions $\varphi\in \mathcal S(\frak a),$
$$ \langle \widehat u^{\,k} ,\varphi \rangle = 
\int_{\frak a} \varphi(\xi) \bigl(\int_{\frak a } E_k(x,-i\xi) dm(x)\bigr) \omega_k(\xi) d\xi \,= \,  \langle \widehat m^{k} , \varphi \rangle. $$
Thus $\widehat m^{k} = \widehat u^{\,k}$ which implies $m=u$ by the injectivity of the Dunkl transform on $\mathcal S^\prime(\frak a).$
\end{proof}

 Consider now a fixed root system $R\subset \frak a$ with two multiplicities $k, \, k^\prime$ satisfying $\text{Re}\, k \geq 0, \text{Re}\, k^\prime \geq 0.$  Then the operator
$$ V_{k^\prime\!,k} := V_{k^{\prime}} \circ V_k^{-1}$$ 
is a topological isomorphism of $\mathcal E(\frak a)$ and intertwines the 
Dunkl operators associated with multiplicities $k$ and $k^\prime$, 
$$ T_\xi(k^\prime) V_{k^\prime\!,k} = V_{k^\prime\!,k}\, T_\xi(k) \quad \text{ for all } \xi \in \frak a.$$
Note that for all $x\in \frak a $ and $z\in \frak a_\mathbb C$, 
\begin{equation}\label{E_V} E_{k^\prime}(x,z) = V_{k^\prime\!,k}\bigl(E_k(\,.\,,z)\bigr)(x).\end{equation}
For fixed $x\in \frak a$ the assignment $\,\langle u_x^{k^\prime\!, k}, \varphi\rangle:= V_{k^\prime\!,k}\, \varphi(x)\,$ defines a compactly supported
distribution $u_x^{k^\prime\!, k}\in \mathcal E^\prime(\frak a)$ satisfying 
\begin{equation} \label{dist_Sonine} \langle u_x^{k^\prime\!,k}, E_k(\,.\,, z)\rangle = E_{k^\prime}(x,z).\end{equation}

\begin{lemma}\label{support} The support of $u_{x}^{k^\prime\!,k}$ is contained in the closed ball $\overline{B}_{|x|}(0).$ 
\end{lemma}

\begin{proof} Let $\varphi\in \mathcal D(\frak a)$ with $\text{supp}\, \varphi 
\cap \overline{B}_{|x|}(0) = \emptyset.$ Then by Proposition \ref{deJeu}, 
$(V_{k^\prime}\circ V_k^{-1})(\varphi)$ vanishes on $B_{|x|}(0)$ and therefore $\langle u_x^{k^\prime\!,k}, \varphi\rangle =0.$ 
 
\end{proof}

\begin{lemma} For $\text{Re}\, k \geq 0, \, \text{Re}\, k^\prime \geq 0$ the following are equivalent. 
\begin{enumerate}
\item[\rm{(1)}] $V_{k^\prime\!,k}$ is positive on $\mathcal P,$  i.e. $V_{k^\prime\!,k} \,p \geq 0$ on $\frak a$ for all $p \in \mathcal P$ 
with $p\geq 0$ on $\frak a.$
\item[\rm{(2)}] $V_{k^\prime\!,k}$ is positive on $\mathcal E(\frak a),$ i.e. $V_{k^\prime\!,k} \,f \geq 0$ for all 
$f \in \mathcal E(\frak a)$ 
with $f\geq 0$.
\item[\rm{(3)}] For each $x\in \frak a$ there exists a probability measure $\mu_x^{k^\prime\!,k}\in M^1(\frak a) $ such that 
\begin{equation}\label{Sonine_Dunkl_L} E_{k^\prime}(x, iy) = \int_{\frak a} E_k(\xi,iy) d\mu_x^{k^\prime\!,k}(\xi) \quad \forall y\in \frak a.
\end{equation}
\end{enumerate}
In this case, the representing measure $\mu_x^{k^\prime\!,k}$ is unique and compactly supported. 
\end{lemma}

\begin{proof}
$(1) \Rightarrow (2)\,$ Let $f\in \mathcal E(\frak a)$ with $f > 0$ on $\frak a$. As $\mathcal P$ is 
dense in $\mathcal E(\frak a)$ (see \cite[Chap.15]{Tr}), there exists a sequence $(p_n) $
consisting of real-valued polynomials $p_n\in \mathcal P$ such that $p_n \to \sqrt f\,$ in $\mathcal E(\frak a).$ Then the polynomials 
$q_n:= p_n^2$ are nonnegative on $\frak a$ and $q_n\to f$ in $\mathcal E(\frak a).$ 
It follows that for alll $x\in \frak a,$
$$ V_{k^\prime\!,k}f(x) = \langle u_x^{k^\prime\!,k}, f\rangle = \lim_{n\to \infty} \langle  u_x^{k^\prime\!,k},
q_n\rangle = \lim_{n\to\infty} V_{k^\prime\!,k}\, q_n(x) \geq 0.$$
A simple approximation argument, using $V_{k^\prime\!,k} 1 = 1$, implies the assertion.

$(2) \Rightarrow (3)\,$ By assumption, the distribution $u_x^{k^\prime\!, k}$ is positive and therefore given by a compactly supported positive 
Radon measure (\cite[Theorem 2.1.7]{H}). Denoting this measure by $\mu_x^{k^\prime\!,k},$ we obtain from \eqref{E_V} that  
$$ E_{k^\prime}(x,z) = \int_{\frak a} E_k(\xi,z) d\mu_x^{k^\prime\!,k}(\xi), \quad z \in \frak a_{\mathbb C}. $$
As $E_k(\xi, 0) = 1,$ evaluation at $z=0$ shows  that $\mu_x^{k^\prime\!,k}$ is a probability measure. 

$(3) \Rightarrow (1)$ In view of formula \eqref{dist_Sonine},  Corollary \ref{laplace_injectivity}(2) implies that
$u_x^{k^\prime\!,k} = \mu_x^{k^\prime\!,k}$. In particular, $\mu_x^{k^\prime\!,k} $ is 
compactly supported and uniquely determined by  \eqref{Sonine_Dunkl_L}.
We claim that $V_{k^\prime\!,k}$ acts on $\mathcal P$  by 
\begin{equation}\label{integralop} V_{k^\prime\!,k}\, p(x) = \int_{\frak a} p(\xi)\, d\mu_x^{k^\prime\!,k}(\xi).\end{equation}
 After replacing $p$ by $V_kp$, it suffices to prove that
 $$ V_{k^\prime} p(x) = \int_{\frak a} V_kp(\xi) \,d\mu_x^{k^\prime\!, k}(\xi) \quad \forall p \in \mathcal P.$$
 But homogeneous expansion in \eqref{Sonine_Dunkl_L} gives
$$ \sum_{n=0}^\infty \frac{1}{n!} V_{k^\prime}(\langle \,.\,, iy\rangle^n)(x) = \sum_{n=0}^\infty \frac{1}{n!}\int_{\frak a} 
V_k(\langle \,.\,, iy\rangle^n)(\xi) d\mu_x^{k^\prime\!,k}(\xi).$$
Comparison of the homogeneous parts in $y$ shows that
$$ V_{k^\prime} (\langle\,.\,,y\rangle^n)(x) = \int_{\frak a } 
V_k(\langle \,.\,,y\rangle^n)(\xi)d\mu_x^{k^\prime\!,k}(\xi)$$
for all $y\in \frak a $ and $n\in \mathbb Z_+ = \{0,1,2\ldots\}.$ 
This implies the assertion.
\end{proof}

The following analyticity result will be important in the next section.

\begin{lemma}\label{intertwiner_analytic} Let $\varphi \in \mathcal S(\frak a).$ Then
for fixed $x \in \frak a$ and $k\in \mathcal K$ with $\text{Re}\,k\geq 0$, the mapping
$\, k^\prime \mapsto V_{k^\prime\!,k}\, \varphi(x) \,$ is analytic on $\{k^\prime \in \mathcal K: {\rm Re}\, k^\prime > 0\}.$

\end{lemma}

\begin{proof} By Proposition \ref{deJeu},
\begin{equation}\label{Formel_Sonine} V_{k^\prime\!,k}\,\varphi(x) = \frac{c_{k^\prime}^2}{c_k^2} 
\mathcal D_{k^\prime}^{-1}\bigl(\omega_k \omega_{k^\prime}^{-1}\mathcal D_k\bigr)\varphi(x) = 
\frac{1}{c_k^2} \int_{\frak a}  \widehat \varphi^{\,k}(\xi) E_{k^\prime}(ix,\xi) \omega_k(\xi)d\xi.\end{equation}
As $\widehat \varphi^{\,k}$ belongs to $\mathcal S(\frak a)$ and $\,k^\prime \to E_{k^\prime}(ix,\xi)$ is analytic on $\{k^\prime: \text{Re}\, k^\prime >0\}$ with $|E_{k^\prime}^B(ix,\xi)|\leq \sqrt{|W|}$, it follows by standard arguments (dominated convergence and Morera's theorem)  that the integral in \eqref{Formel_Sonine} depends analytically on  $k^\prime$ with $\text{Re}\, k^\prime >0.$
\end{proof}

\begin{remark}\label{Sonine_Soltani}
In the rank one case, Y. Xu derived in \cite[Lemma 2.1]{X} for $k^\prime > k >0$ the explicit formula
$$ V_{k^\prime\!,k} f(x) = \frac{\Gamma(k^\prime + 1/2)}{\Gamma(k^\prime -k)\Gamma(k+1/2)} \int_{-1}^1 f(xt) |t|^{2k}(1+t)(1-t^2)^{k^\prime-k-1} dt.$$
This operator and the associated
Sonine type integral representation for the rank one Dunkl kernel,
$$ E_{k^\prime}(x,z) = \frac{\Gamma(k^\prime + 1/2)}{\Gamma(k^\prime -k)\Gamma(k+1/2)} \int_{-1}^1 E_k(xt,z) |t|^{2k}(1+t)(1-t^2)^{k^\prime-k-1} dt$$
were further studied in \cite{Sol}. 

\end{remark}

\section{The Sonine formula for Bessel functions of type $B_n$}\label{B_n}

In this section, we consider the Dunkl kernel $E_k^B$ and the Bessel function $J_k^B$ associated with root system 
$$ B_n = \{ \pm e_i, 1\leq i \leq n\} \cup \{\pm e_i \pm e_j, \, 1\leq i < j \leq n\}\subset \mathbb R^n, $$ 
where $\mathbb R^n$ is equipped with its usual inner product. 
The associated reflection group is the hyperoctahedral group $W(B_n) = S_n \ltimes \mathbb Z_2^n, $ and
the multiplicity is of the form $k=(k_1, k_2)$ where $k_1$ and $k_2$ denote the value on the roots $\pm e_i$ and $\pm e_i\pm e_j$ respectively. 
We shall derive an explicit Sonine formula for the Bessel function $J_k^B$ 
at the reference point $\,{\bf 1} =(1, \ldots, 1)\in \mathbb R^n,$ extend it in a distributional sense to larger classes of multiplicities 
and construct counterexamples where the associated Bessel functions have no Sonine formula. Of decisive importance for our calculations is the well-known fact that 
$J_k^B$ can be expressed in terms of a certain multivariable  hypergeometric function. To recall this, we need some further notation.

Fix some index $\alpha >0.$ For partitions  $\lambda = (\lambda_1, \ldots, \lambda_n)\in \mathbb Z_+^n, \lambda_1 \geq 
\ldots \geq \lambda_n$ 
(for short, $\lambda \geq 0$) we denote by $C_\lambda^\alpha$ the Jack polynomials of index $\alpha$ in $n$ variables (c.f. \cite{Sta}), normalized such that
$$ (z_1 + \ldots + z_n)^m = \sum_{|\lambda|=m} C_\lambda^\alpha(z)\quad \text{ for all } m \in \mathbb Z_+.$$

Following the notation of \cite{K} and \cite{BF1},   we define for $\mu \in \mathbb C$ with 
$\text{Re}\, \mu > \frac{1}{\alpha}(n-1)$
the hypergeometric function 
$$ _0F_1^\alpha(\mu; z,w):= \sum_{\lambda\geq 0} \frac{1}{(\mu)_\lambda^\alpha |\lambda|!} \cdot \frac{C_\lambda^\alpha(z) 
C_\lambda^\alpha(w)}{C_\lambda^\alpha({\bf 1})} \quad (z,w\in \mathbb C^n)$$
with the generalized Pochhammer symbol 
$$ (\mu)_\lambda^\alpha := \prod_{j=1}^n \bigl(\mu - \frac{1}{\alpha} (j-1)\bigr)_{\lambda_j}.$$
In the one-dimensional  case $n=1$, the Jack polynomials are independent of $\alpha$ and given by $C_\lambda^\alpha(z) = z^\lambda, \, \lambda\in \mathbb Z_+$. Thus
$$ _0F_1^\alpha(\mu;-\frac{z^2}{4},1) = j_{\mu-1}(z).$$

In the general case, the Bessel function $J_k^B$ is expressed in terms of $_0F_1^\alpha$ as follows: 

\begin{proposition}
\label{hypergeom-bessel-B}
 Let $k=(k_1, k_2)$ with $\text{Re}\, k_1 \geq 0$ and $k_2 >0$. Then 
$$ J_k^B(z,w) = \,_0F_1^\alpha\bigl(\mu;\frac{z^2}{2},\frac{w^2}{2}\bigr) =\,_0F_1^\alpha\bigl(\mu;\frac{z^2}{4},w^2\bigr)\quad (z,w\in \mathbb C^n),$$
where $\alpha = \frac{1}{k_2}, \,\, z^2:= (z_1^2, \ldots, z_n^2),$ and
$\mu = \mu(k) := k_1 + k_2(n-1) + \frac{1}{2}$. 
\end{proposition}

\begin{proof} See \cite[Propos. 4.5]{R2} and \cite[Section 6]{BF1} 
for real $k_1 \geq0.$ The general case follows by analytic continuation.
\end{proof}


The key to the subsequent Sonine type integral representation for the Bessel function $J_k^B$ is Kadell's generalization of the Selberg integral.
For parameters $\kappa, \mu, \nu  \in \mathbb C$ with $\text{Re}\, \kappa\geq 0$ and $\,\text{Re} \, \mu,  \text{Re}\, \nu > 
\text{Re}\, \kappa (n-1),$ the Selberg integral   is given by 
\begin{align}\label{Selberg-int} \int_{]0,1[^n} &\prod_{j=1}^n  x_j^{\mu-\kappa(n-1)-1}(1-x_j)^{\nu-\kappa(n-1)-1}\!
\prod_{1\leq i<j\leq n} |x_i-x_j|^{2\kappa} dx\notag \\
& =\, \prod_{j=1}^n \frac{\Gamma(1+\kappa j)}{\Gamma(1+\kappa)}\cdot 
\prod_{j=1}^n \frac{\Gamma(\mu-\kappa(j-1))\Gamma(\nu-\kappa(j-1))}{\Gamma(\mu+\nu -\kappa(j-1))} \,:= I_n(\kappa, \mu, \nu)
\end{align}
(see e.g \cite{FW}). With the normalized Selberg density
$$ s_{\mu,\nu}^\kappa(x) := \frac{1}{I_n(\kappa,\mu,\nu)}\cdot \prod_{j=1}^n x_j^{\mu-\kappa(n-1)-1}(1-x_j)^{\nu-\kappa(n-1)-1}\!
\prod_{1\leq i<j\leq n} |x_i-x_j|^{2\kappa}, $$  
Kadell's \cite{Ka} generalization  of the Selberg integral (c.f. also \cite[(2.46)]{FW}) reads
\begin{equation}\label{Kadell} \int_{]0,1[^n} \frac{C_\lambda^\alpha (x)}{C_\lambda^\alpha({{\bf 1}})}\, s_{\mu,\nu}^{1/\alpha}(x)dx \, = \, 
\frac{(\mu)_\lambda^\alpha }{(\mu+\nu)_\lambda^\alpha}.\end{equation}
Formula  
\eqref{Kadell}  implies that for $z\in \mathbb C^n$ and $ \mu, \nu \in \mathbb C$ with $\,\text{Re} \, \mu,  \text{Re}\, \nu > \frac{1}{\alpha}(n-1),$
\begin{equation}\label{Sonine} _0F_1^\alpha(\mu+\nu;z,{\bf 1}) = \int_{]0,1[^n} \!\!\,_0F_1^\alpha(\mu; z,x)\,s_{\mu, \nu}^{1/\alpha}(x)dx.\end{equation}
This is a Sonine formula for $\, _0F_1^\alpha$; in case $n=1$ it reduces to the classical 
Sonine integral \eqref{intrep-1-dim} for one-variable Bessel functions. 

The Sonine formula \eqref{Sonine} translates to the Bessel function of type $B_n$ as follows: 
Let $k = (k_1, k_2)$ with $\text{Re}\, k_1 \geq 0$ and $k_2>0.$ For $h \in \mathbb C$ put 
$$k^\prime(h):= (k_1+h, k_2).$$
Then for $h\in \mathbb C$ 
with $\text{Re}\, h > k_2(n-1)$ and all $z\in \mathbb C^n$,
\begin{equation}\label{density} J_{k^\prime(h)}^B(z,{\bf 1}) = \int_{]0,1[^n} J_k^B(z,x) f_{k,h}(x)dx\end{equation}
with the density
\begin{equation}\label{density_2} f_{k,h}(x) = \frac{2^n}{I_n(k_2, \mu(k),h)} \prod_{j=1}^n (x_j^2)^{k_1} (1-x_j^2)^{h-k_2(n-1)-1} 
\prod_{i<j} |x_i^2-x_j^2|^{2k_2}\end{equation}
and with $\mu(k)$ as in Proposition \ref{hypergeom-bessel-B}.
Note that $f_{k,h}$ is $W(B_n)$-invariant, and therefore
\begin{equation} \label{hilfsdarst} J_{k^\prime(h)}^B(z,{\bf  1}) = \int_{]0,1[^n} E_k^B(z,x) f_{k,h}(x)dx.\end{equation} 
We extend $f_{k,h}$ by zero to a measurable function on $\mathbb R^n$. For $\text{Re}\, h>k_2(n-1)$ we have $f_{k,h}\in L^1_{loc}(\mathbb R^n),$ 
which corresponds to a complex Radon measure
$$d\rho_{k,h}(x) = f_{k,h}(x)dx.$$ 

Now recall from Section \ref{intertwining} the distributions 
$u_x^{k^\prime\!, k}\in \mathcal E^\prime(\mathbb R^n)$ defined by $\, \langle u_{x}^{k^\prime\!,k}, \varphi\rangle = V_{k^\prime\!,k}\, \varphi(x).$

\begin{definition} Consider $k=(k_1, k_2)$ on root system $B_n$ with $\text{Re}\, k_1 \geq 0$ and $k_2 >0$ as above.
For  $h\in \mathbb C$ with $\text{Re}\, h \geq -\text{Re}\,k_1$ denote
by  $S_{k,h}\in \mathcal E^\prime (\mathbb R^n)$ the $W(B_n)$-mean of  $u_{\bf 1}^{k^\prime(h), k},$ i.e.
$$ \langle S_{k,h}, \varphi\rangle := \frac{1}{2^n n!}\!\sum_{w\in W(B_n)}\! \langle u_{\bf 1}^{k^\prime(h),k}, w.\varphi \rangle.$$

\end{definition}

According to Lemma \ref{support}, $S_{k,h}$ is supported  in the Euclidean ball $B_{\sqrt n}(0),$ and it is $W(B_n)$-invariant. Thus in view of 
\eqref{dist_Sonine}, we have the following distributional extension of the Sonine formula \eqref{density}: 
\begin{equation}\label{Sonine_rep}
J_{k^\prime(h)}^B(z, {\bf  1}) = \langle S_{k,h}, J_k^B(z, \,.\,)\rangle = \langle S_{k,h}, E_k^B(z, \,.\,)\rangle,\quad z \in \mathbb C^n.
\end{equation}

The next result is a simple criterion for the existence of a Sonine-type integral representation 
for the  Bessel function of type $B_n$.

\begin{proposition} Let $k = (k_1, k_2)$ with $k_2> 0$ and $\text{Re}\, k_1 \geq 0.$ Then for  ${\rm Re}\, h\geq -{\rm Re}\,k_1\,$ the following are equivalent:
\begin{enumerate}
\item[\rm{(1)}] The distribution $S_{k,h}$ is a complex (positive) measure.   
\item[\rm{(2)}] There exists a bounded complex (positive) Radon 
measure $m\in M_b(\mathbb R^n)$ such that the following Sonine formula holds:
$$ J_{k^\prime(h)}^B(i\xi, {\bf 1}) = 
\int_{\mathbb R^n} J_{k}^B (i\xi, x)\, dm(x) \quad \text{for all } \xi \in \mathbb R^n.$$ 
\end{enumerate}
In this case, the measure $m$ in (2) is unique and given by $\,m = S_{k,h}.$ 

\end{proposition}

\begin{proof} The implication $(1) \Rightarrow (2)\,$ is immediate from identity \eqref{Sonine_rep}. For the converse direction, note first
that we may assume that $m$ is $W(B_n)$-invariant.
 Thus from \eqref{Sonine_rep}, we obtain that
 $$ \langle S_{k,h}\,,  E_k^B(i\xi, \,.\,)\rangle = \langle m, J_k^B(i\xi, \,.\,)\rangle = \langle m,  E_k^B(i\xi, \,.\,)\rangle 
 \quad \text{for all } \xi \in \mathbb R^n.$$
Corollary \ref{laplace_injectivity}(2) for the Laplace transform  now implies that $\,m = S_{k,h}$. 
 \end{proof}
 
Identity \eqref{Sonine_rep}
together with \eqref{hilfsdarst} and 
the injectivity of the Dunkl  Laplace transform (Corollary \ref{laplace_injectivity})  imply that for $\text{Re} \, h>k_2(n-1),$
$$ S_{k,h} = \rho_{k,h}.$$ 
We are interested to know for which range of $h$ the distribution $S_{k,h}$ is actually a complex Radon measure, i.e. of order zero. 
The following useful observation of Sokal \cite[Lemmata 2.1, 2.2 and
Proposition 2.3]{S}) will provide a necessary condition.


\begin{lemma}\label{Sokal}
Let $\Omega\subseteq \mathbb R^n$ be open and $D\subseteq \mathbb C$ open and connected. Suppose that 
$$F:\Omega\times D\to\mathbb C, \quad (x,\lambda)\to f_\lambda(x):=F(x,\lambda)$$
is a continuous function such that $F(x,.)$ is analytic on $D$ for
 each $x\in \Omega$. Extend $f_\lambda$ by zero to all of $\mathbb R^n$  and define $\,u_\lambda \in \mathcal D^{\prime}(\Omega)$ by
 \[ \langle u_\lambda,\varphi\rangle = \int_\Omega \varphi(x) f_\lambda(x) dx.\]
 Then the following hold:
\begin{enumerate}\itemsep=-1pt
 \item[\rm{(1)}] The map $\,\lambda\mapsto u_\lambda\,, \,\, D \to 
 \mathcal D^\prime(\Omega)  $ is weakly analytic, which means that  
$\lambda\mapsto \langle u_\lambda, \varphi\rangle $
is analytic for all $\varphi \in\mathcal D(\Omega)$.
\item[\rm{(2)}]  Let  $D_0\subseteq D$ be a nonempty open set, and suppose that there is a weakly  analytic map 
  $\lambda\mapsto \widetilde u_\lambda, \,\, D \to \mathcal D^\prime(\mathbb R^n)$  
such that for each $\lambda\in D_0$ the distribution $\widetilde u_\lambda$ extends the distribution  $u_\lambda$ from $\Omega$ to  $\mathbb R^n$. Then  $\widetilde u_\lambda$ extends $u_\lambda$
for each $\lambda\in D$. 
Moreover, if $\,\widetilde u_\lambda$ is a complex Radon measure on $\mathbb R^n$, then  $f_\lambda$ belongs to
$L^1_{loc}(\overline\Omega)$. This means that $f_\lambda$ is integrable over a sufficiently small neighborhood in $\mathbb R^n$ of any point $x\in \overline \Omega.$
 \end{enumerate}
\end{lemma}

Let us mention at this point that for certain values of $\lambda\in D\setminus D_0$, it may happen that $f_\lambda$ is identical zero while $\widetilde u_\lambda$ is a nonzero measure concentrated on
the boundary of $\Omega$. A typical example are the Riesz distributions on symmetric cones (see \cite{S}), where this phenomenon occurs in the discrete points of the Wallach set. 

\medskip
To apply Lemma \ref{Sokal} to our situation, 
fix $k=(k_1, k_2)$ with $\text{Re}\, k_1\geq 0$ and $k_2 >0$ and put 
$$ D:= \{ h\in \mathbb C: \text{Re}\, h> - \text{Re}\, k_1\}, \quad D_0:= \{h\in \mathbb C: \text{Re}\, h > k_2(n-1)\}.$$

\begin{theorem} \label{main}
\begin{enumerate}
 \item The mapping $h \mapsto S_{k,h}$ is weakly analytic on $D$. 
 \item Let $h\in D$ and suppose that $S_{k,h}$ is a complex Radon measure on $\mathbb R^n$.  Then either $h \in D_0,$ in which case $S_{k,h}= \rho_{k,h}$, or 
  $h$ is contained in the discrete set $\, \{0,k_2, \ldots, k_2(n-1)\} - \mathbb Z_+.$ 
  \item Suppose that $k_1$ is real and $S_{k,h}$ is a positive Radon measure, then in addition to the condition in (2), $h$ must be real.

  \end{enumerate} 
 
\end{theorem}

\begin{proof} (1) Let $\varphi \in \mathcal D(\mathbb R^n)$. By Lemma \ref{intertwiner_analytic}, the 
mapping $\,k^\prime \mapsto V_{k^\prime\!, k}\,\varphi({\bf  1}) = u_{\bf 1}^{k^\prime\!,k}(\varphi)$ is analytic on 
$\{\text{Re}\, k^\prime >0\},$ and therefore
$\,\langle S_{k,h}, \varphi\rangle $ depends analytically on $h\in D.$ 

(2) Consider the integral formula  \eqref{density} with density \eqref{density_2}, which  is valid for $h\in D_0$.   The function  
 $$h\mapsto I_n(k_2, \mu(k),h)= \prod_{j=1}^n \frac{\Gamma(1+jk_2)}{\Gamma(1+k_2)} \cdot \prod_{j=0}^{n-1} 
   \frac{\Gamma\bigl(k_1+jk_2 +\frac{1}{2}\bigr)\Gamma(h-jk_2)} {\Gamma\bigl(k_1+jk_2 +h + \frac{1}{2}\bigr)}$$
 (c.f. \eqref{Selberg-int}) extends to a meromorphic function on $D$ without zeroes and with pole set
 $$ D\cap\bigl(\{ 0, k_2, \ldots, k_2(n-1)\} - \mathbb Z_+\bigr).$$ Thus the function $h\mapsto f_{k,h}(x)$ extends analytically 
 to $D$ for each $x\in ]0,1[^n.$ 
 If $h\in D$ is such that $S_{k,h}$ is a complex Radon measure on $\mathbb R^n$,  then it follows from Lemma \ref{Sokal} that $f_{k,h}\in L^1_{loc}([0,1]^n).$ 
 This in turn is satisfied exactly if one of the following two conditions is fulfilled:
 \begin{itemize}\itemsep=-1pt 
  \item[\rm{(i)}]  $\text{Re}\, h > k_2(n-1);$
  \item[\rm{(ii)}] $h$ is a pole of the function $I_n(k_2, \mu(k),\,.\,),$ in which case $f_{k,h}$ is identical zero on $]0,1[^n.$
 \end{itemize}

 
Finally, part (3) is immediate from part (2).
 
\end{proof}

\begin{remark} In the situation of part (2)(ii), where $f_{k,h}$ is identical zero, the distribution $S_{k,h}$ is not equal to zero, which follows from \eqref{Sonine_rep} and the fact that $J_{k^\prime(h)}^B(0,{\bf 1}) = 1$. Thus $S_{k,h}$ is no longer represented by $f_{k,h}$. We conjecture that for $h\in \{0,k_2, \ldots, k_2(n-1)\},$  the distribution $S_{k,h}$  
is a (positive) measure which is supported in the boundary of $[0,1]^n$, and that also  the intertwiner $V_{k^\prime(h),k}$ is positive for these discrete values of $h$. 
\end{remark}

As an important  consequence of the previous results, we obtain that in the $B_n$-case
and for arbitrary multiplicity $k$ with $k_1 \geq 0$ and $k_2>0$  there 
exist multiplicities $k^\prime = (k_1 +h, k_2) \geq k$ such that the Bessel function $J_{k^\prime} ^B$ has no Sonine integral representation 
with respect
to $J_k^B,$ and that the intertwiner $V_{k^\prime\!, k}$ is not positive. More precisely, the following holds.

\begin{corollary}\label{main_Bessel} Let $k=(k_1, k_2)\in \mathbb C^2$ with  $k_2 >0, \, \text{Re}\, k_1 \geq 0$ 
and consider $k^\prime = (k_1+h,k_2)$ with $\text{Re}\, h > -\text{Re}\, k_1.$ 
\begin{enumerate}
\item Suppose that there exists a bounded complex Radon measure $m\in M_b(\mathbb R^n)$ such that the Sonine formula
\begin{equation}\label{Sonine_1} J_{k^\prime}^B(i\xi, {\bf 1}) = \int_{\mathbb R^n}  J_{k}^B(i\xi,x) dm(x)\end{equation}
holds for all $\xi \in \mathbb R^n$. Then either $\text{Re}\, h>k_2(n-1),$ in which case \eqref{Sonine_1} holds with  $m = \rho_{k,h}$, or $\,h $ is contained in  $\,\{0, k_2, \ldots, k_2(n-1)\} -\mathbb Z_+\,.$ 
If $k_1 \geq 0$ and $m$ is positive, then in addition $h$ must be real. 
\item Suppose that $k_1\geq 0$ and that $V_{k^\prime\!, k}$ is positive. Then 
$h$ is contained in the set
$$ \Sigma(k_2):= \,]k_2(n-1), \infty [ \, \cup \, \bigl(\{0, k_2, \ldots, k_2(n-1)\} -\mathbb Z_+\bigr).$$
\end{enumerate} 
\end{corollary}

\begin{remarks}
(1)
The set $\Sigma(k_2)$ is closely related with the so-called Wallach set
$$ ]d(n-1), \infty [ \, \cup \, \{0, d, \ldots, d(n-1)\},$$
where $d\in \mathbb N$ is the  Peirce constant of a symmetric cone. The Wallach set plays an important role in the 
analysis on symmetric cones, see \cite{FK} for some background. It describes the set of parameters for which
 Riesz distributions on a symmetric cone are actually positive measures, a result which is due to Gindikin \cite{G}.

(2) Corollary \ref{main_Bessel} should be compared with the results of \cite[Section 4]{RV2} for Bessel functions on symmetric cones, which are closely related to the Bessel functions $J_k^B$. In \cite[Theorem 4]{RV2} 
also a sufficient condition for the existence of Sonine formulas between Bessel functions on a symmetric cone is given. 
It is based on the knowledge of the parameters for which Riesz distributions are actually measures. By similar methods, based on the recent results \cite{R3} about Riesz distributions in the Dunkl setting,
it should be possible to obtain extended parameter ranges for positive Sonine formulas of type  \eqref{Sonine_1}, but this would be not strong enough to imply positivity of the associated intertwining operator. 

 \end{remarks}

 \section{Consequences for hypergeometric functions and Heckman-Opdam polynomials of type BC}\label{trig}

 We start with some basic facts from Heckman-Opdam theory, see \cite{HS, O2} for more details.
 Let again $(\frak a, \langle\,.\,,\,.\,\rangle) $ be a finite dimensional Euclidean space, which we identify with its dual $\frak a^*= \text{Hom}(\frak a, \mathbb R)$ via the given inner product. 
 Let $R$ be a crystallographic, not necessarily reduced root system in $\frak a$ with associated reflection group $W$  and 
 fix a positive subsystem $R_+$ of $R$ as well as a $W$-invariant multiplicity function $k$ on $R$, where we assume  for simplicity that $k$ is real-valued with $\, k\geq 0.$ 
 The Cherednik operators associated with $R_+$ and $k$ are defined by
 $$ D_\xi(k) = \partial_\xi + \sum_{\alpha \in R_+} k(\alpha) \langle \alpha, \xi\rangle \frac{1}{1-e^{-\alpha}} (1-\sigma_\alpha) -\langle \rho(k), \xi\rangle, \quad \xi \in \mathbb R^n$$
 where $e^{\lambda}(z) := e^{\langle \lambda,z\rangle}$ for $\lambda, z \in \frak a\, $ and  
 $\, \rho(k) = \frac{1}{2} \sum_{\alpha \in R_+} k(\alpha) \alpha.$
 
  The $D_\xi(k), \xi \in \frak a$ commute, and for each $\lambda \in \frak a _\mathbb C$ there exists a unique analytic function $G(\lambda,k; .\,)$ on a common $W$-invariant tubular neighborhood of $\frak a$ in $\frak a_{\mathbb C}$, the Opdam-Cherednik kernel,  satisfying
  $$ D_\xi(k)G(\lambda,k;\,.\,) = \langle \lambda, \xi \rangle\, G(\lambda,k; \,.\,) \>\> \forall \, \xi \in \frak a ; \,\,G(\lambda, k; 0) = 1.$$
The hypergeometric function associated with $R$ is  defined by
$$ F(\lambda, k;z) = \frac{1}{|W|} \sum_{w\in W} G(\lambda, k; w^{-1}z).$$
It is $W$-invariant in both $z$ and $\lambda$.
Closely related with the hypegeometric function are the Heckman-Opdam polynomials. To introduce these, write
$\alpha^\vee = \frac{2\alpha}{\langle \alpha, \alpha \rangle}$ for $\alpha \in R$ and consider the weight lattice and the set of dominant weights associated with $R$ and $R_+$, 
$$ P = \{\lambda \in \frak a: \langle \lambda, \alpha ^\vee \rangle \in \mathbb Z \>\> \forall \alpha \in R\,\}; \quad
P_+ = \{\lambda \in P: \langle \lambda, \alpha^\vee\rangle \geq 0\,\, \forall \alpha \in R_+\,\}.$$ 
Note that $R_+\subset P_+$. We equip $P_+$ with the usual dominance order, that is, $\mu < \lambda$ iff $\lambda-\mu$ is a sum of positive roots. 
Denote further by
$\, \mathcal T:= \text{span}_{\mathbb C}\{e^{i\lambda}, \, \lambda \in P\}\,$
the space of trigonometric polynomials associated with $R$. Notice that the members of $\mathcal T$  are $2\pi Q^\vee$-periodic, where
$ Q^\vee = \text{span}_{\mathbb Z}\{\alpha^\vee, \, \alpha \in R\},$   and that the orbit sums 
$$ M_\lambda = \sum_{\mu \in W\!\lambda} e^{i\mu}\,, \quad \lambda \in P_+$$
form a basis of the subspace $\mathcal T^W$ of $W$-invariant elements from $\mathcal T.$ Consider the compact torus $\mathbb T = \frak a/2\pi Q^\vee$ with the weight function 
$$ \delta_k(t) := \prod_{\alpha \in R_+} |\sin\frac{\langle \alpha, t\rangle}{2}|^{2k_\alpha}.$$
The Heckman-Opdam polynomials associated with $R_+$ and $k$ are defined by
$$ P_\lambda(k;z) := M_\lambda(z) + \sum_{\nu < \lambda} c_{\lambda\nu}(k) M_\nu(z); \quad \lambda \in P_+\,, z\in \frak a_{\mathbb C}$$
where the coefficients $c_{\lambda\nu}(k)\in \mathbb R$ are uniquely determined by the condition that $P_\lambda(k;\,.\,)$ is orthogonal to $M_\nu$ in $L^2(\mathbb T, \delta_k)$ for all $\nu \in P_+$ with $\nu<\lambda.$ It is known that the coefficients actually satisfy $c_{\lambda\nu}(k) \geq 0$ for all indices ${\lambda,\nu}$ (\cite[Par.11]{M}), and that 
the family $\{P_\lambda(k; \,.\,), \lambda\in P_+\,\}$ forms an orthonormal basis of $L^2(\mathbb T, \delta_k)^W$, the subspace of $W$-invariant functions from $L^2(\mathbb T, \delta_k)$. 
The renormalized polynomials
$$ R_\lambda(k;z):= \frac{P_\lambda(k;z)}{P_\lambda(k;0)}$$
are related with the hypergeometric function via (see \cite{HS}) 
$$ R_\lambda( k;z) = F(\lambda+\rho(k),k;iz).$$
As $R_0(k,\,.\,) =1$, it follows that 
\begin{equation}\label{normalization} F(\rho(k),k;\,.\,) =1.\end{equation}

We now consider $\mathfrak a = \mathbb R^n$ with the nonreduced root system
 $$ R= BC_n = \{\pm e_i, \pm 2 e_i, 1\leq i \leq n\}\cup\{ \pm(e_i \pm e_j), 1\leq i < j \leq n\} \subset \mathbb R^n. $$
 Its weight lattice is $P = \mathbb Z^n$ and the torus $\mathbb T$ is given by $\mathbb T= (\mathbb R/2\pi \mathbb Z)^n.$
 We write mulitplicities on $R$ as $k=(k_1,k_2, k_3)$ with $k_1, k_2, k_3$ the values on the roots $e_i, 2e_i, e_i \pm e_j$. If $n=1$, then 
$k_3$ does not appear and $k=(k_1,k_2)$.
 We fix some positive subsystem $R_+$ and denote 
 the associated Opdam-Cherednik kernel and hypergeometric function
 by $G_{BC}$ and $F_{BC}.$

 Dunkl operators are scaling limits of Cherednik operators, which implies that Dunkl 
kernels and Bessel functions can be obtained by a contraction limit from Opdam-Cherednik kernels and hypergeometric functions. We shall need the following variant of Theorem 4.12 in \cite{dJ2} (see also \cite{RV1}) which was originally formulated for reduced root systems. The proof extends to $R=BC_n$  in the obvious way. 

\begin{lemma}\label{scaling} Consider the root systems $BC_n$ with multiplicity $k=(k_1,k_2,k_3)$ and $B_n$ with
multiplicity $k_0:= (k_1+k_2, k_3).$ Let further $K, L \subset \mathbb C^n$ be compact, $\delta>0$ some constant and let $h:(0, \delta) \times L \to \mathbb C^n$  a continuous function such that $\, \lim_{\epsilon \to 0}  \epsilon h(\epsilon, \lambda) = \lambda \,$ uniformly on $L$. Then
\begin{equation}\label{contraction} \lim_{\epsilon \to 0} G_{BC}(h(\epsilon, \lambda),k; \epsilon z) = 
E_{k_0}^B(\lambda,z), \quad \lim_{\epsilon \to 0} F_{BC}(h(\epsilon, \lambda),k; \epsilon z) = 
J_{k_0}^B(\lambda,z) \end{equation}
uniformly for $(\lambda,z)\in L\times K.$ 

\end{lemma}

Hypergeometric functions associated with root systems generalize the spherical functions of Riemannian symmetric spaces $G/K$  of noncompact type. More precisely, 
suppose that $\Sigma$ is the restricted root system of  $G/K$ with Weyl group $W$ and geometric multiplicities $m_\alpha, \, \alpha \in \Sigma$. 
Let $F$ be the hypergeometric function associated with $R= 2\Sigma$ and define the multiplicity $k$ on $R$ by $k(2\alpha) := \frac{1}{2}m(\alpha).$  Consider the decomposition $G=KAK$ and let $\frak a := Lie(A)$, which is a Euclidean space with the Killing form $\langle \,.\,, \, . \,  \rangle$.   Then the spherical functions of $G/K$, considered as $W$-invariant functions on $\frak a$, are given by
$ \,\varphi_\lambda(x) = F(\lambda,k;x), \,\lambda \in \frak a_{\mathbb C}$. From the Harish-Chandra formula \cite[Theorem IV.4.3]{Hel} and the Kostant convexity theorem it follows that for $R$ and $k$ as above, 
$$ F(\lambda + \rho(k),k;x) = \int_{C(x)} e^{\langle \lambda,\xi\rangle} dm_x^k(\xi)\quad \forall \lambda \in \frak a_{\mathbb C}$$
where $C(x)\subset \frak a$ again denotes the convex hull of the $W$-orbit of $x$ and $m_x^k$ is a certain $W$-invariant probability measure.
For root system $A_n$ and certain $BC_n$-cases, this integral representation was recently extended in \cite{Sa1, Sa2} to arbitrary 
non-negative multiplicities, including a detailed analysis of the representing measures $m_x^k$. 
See also \cite{Su} for an alternative approach in the $A_n$-case.
 A natural generalization would be an integral representation of Sonine type between hypergeometric functions with different multiplicities, where we allow a constant shift (depending on the multiplicity) in 
 the spectral variable. 
 
 In the following, we consider $R=BC_n$ with $n \geq 2.$ Recall that $F_{BC}$ is $W(B_n)$-invariant both in the spatial and the spectral variable. 
 For $a \in \mathbb R^n$ we shall write 
 $ a \geq 0$ if $a$ is contained in the positive Weyl chamber $[0, \infty[^n$, and for $c \in \mathbb R$ we put ${\bf c}:= (c, \ldots, c) \in \mathbb R^n$.
As a consequence of Corollary \ref{main_Bessel}, we obtain the following result:

\begin{theorem}\label{hypergeom_integral} Fix $k=(k_1,k_2,k_3)\in \mathbb R^3$ with $k_1, k_2 \geq 0, \, k_3 >0$ and consider $k^\prime =k^\prime(h):= (k_1+h_1, k_2+h_2, k_3)$ with 
$ h_1>-k_1, h_2>-k_2.$ Suppose that there are constants  $\sigma(k), \sigma(k^\prime) \in \mathbb R^n$ and $c_0>0$ such that for all 
$c \in \mathbb R$ with $0<c< c_0$ there holds a Sonine formula
\begin{equation}\label{Sonine_hypergeom} F_{BC}(\lambda+ \sigma(k^\prime), k^\prime; {\bf c}) = \int_{\mathbb R^n} 
F_{BC}(\lambda+ \sigma(k),k;\xi)\, dm_c(\xi)\quad \forall \lambda \in \mathbb C^n,\end{equation}
with a positive Radon measure $m_c$ on $\mathbb R^n$ which is supported in $[0,c]^n.$ Then $h_1+h_2$ is contained in the set
$$ \Sigma(k_3)=\, ]k_3(n-1), \infty [ \, \cup \, \bigl(\{0, k_3, \ldots, k_3(n-1)\} -\mathbb Z_+\bigr).$$
In particular,  there are multiplicities $k^\prime \geq k \geq 0$ on $BC_n$ such that a Sonine formula \eqref{Sonine_hypergeom} does not exist.
\end{theorem}

Before turning to the proof of this result, let us mention a canonical analogue of integral representation \eqref{Sonine_hypergeom} in the rank one case $R= BC_{1}.$ 
In this case, the Heckman-Opdam hypergeometric function is given by
$$ F_{BC_1} (\lambda, k;t) = \varphi_{-2i\lambda}^{(\alpha, \beta)}\bigl(t/2) \quad \text{ with }\, \alpha = k_1+k_2 -1/2, \beta = k_2 - 1/2$$
and with the Jacobi functions 
$$ \varphi_\lambda^{(\alpha, \beta)}(t) = \, _2F_1\bigl(\frac{1}{2}(\alpha + \beta + 1 + i\lambda), \frac{1}{2}(\alpha + \beta +1 -i\lambda); \alpha +  1, - \sinh^2 t\bigr), $$
see \cite[Ex.1.3.2]{O2} and \cite[formula (2.4)]{Ko}. Suppose that $\alpha > \beta > -1/2$ and $\alpha > \gamma > \delta > -1/2.$ Then according to identity (5.70) of \cite{Ko}, there holds for each $x\in [0, \infty[$ a  Sonine formula of the form
$$ \varphi_\lambda^{(\alpha, \beta)} (x) = \int_{0}^x \varphi_\lambda^{(\gamma, \delta)} (\xi) d\mu_x(\xi) \quad \forall \lambda \in \mathbb C, $$
with positive measures $\mu_x$ depending on $\alpha, \beta, \gamma, \delta$. 
In terms of $F_{BC_1}$ this means that for multiplicities $k= (k_1, k_2) $ with $k_1, k_2 > 0$ and $k^\prime = (k_1 +h_1, k_2+h_2)$ with $h_1 > -k_1, h_2 > -k_2$ as well as $h_1+h_2 >0,$ 
$$ F_{BC_1}(\lambda, k^\prime;c) = \int_0^c F_{BC_1}(\lambda, k; t)  dm_c(t),\quad \forall \lambda \in \mathbb C$$
with positive measures $ m_c$ depending on $k, k^\prime$ for all $c \in ]0, \infty[.$

\begin{proof}[Proof of Theorem \ref{hypergeom_integral}]
The main idea is to apply Lemma \ref{scaling}.
First, we  check that 
there is some constant $C_0>0$ such that 
\begin{equation}\label{M_bound}  M_c:= m_c(\mathbb R^n) \leq C_0 \quad\text{for all }\, 0< c < c_0.\end{equation}
 For this let $\rho:= \rho(k), \rho^\prime:= \rho(k^\prime),
\sigma:= \sigma(k), \sigma^\prime := \sigma(k^\prime),\, \tau:= \sigma^\prime - \sigma.$  
Consider formula \eqref{Sonine_hypergeom}
 with $\lambda= -\rho-\sigma$. By the normalization \eqref{normalization} and the $W$-invariance of $F$ in the spectral variable, we get
\begin{align*}   F(\rho-\tau, k^\prime;{\bf c}) &=
F(-\rho + \tau, k^\prime; {\boldsymbol c}) =  F(\lambda+\sigma^\prime , k^\prime; {\boldsymbol c})
=\int_{\mathbb R^n} F(-\rho,k;\xi) dm_c(\xi) \notag  \\
& = m_c(\mathbb R^n) = M_c\,.\end{align*}
On the other hand, as $\rho^\prime \geq 0$ we obtain from \cite[Theorem 3.3]{RKV} the estimate
$$ F(\rho-\tau, k^\prime, {\boldsymbol c}) = 
  F(\rho^\prime + (\rho-\rho^\prime -\tau), k^\prime, {\boldsymbol c}) \, \leq 
  F(\rho^\prime,k^\prime;{\boldsymbol c}) \cdot e^{\max_{w\in W} \langle w(\rho-\rho^\prime - \tau), {\boldsymbol c}\rangle}.$$
  Therefore
  $$ M_c = F(\rho-\tau, k^\prime, {\boldsymbol c}) \leq 
e^{c \cdot\max_{w\in W} \langle w(\rho-\rho^\prime - \tau), {\bf 1}\rangle },$$
which yields the claimed boundedness \eqref{M_bound} of the masses $M_c$.

For  $0<c <c_0$ denote by $\widetilde m_c$ the image measure of $m_c$ under the
mapping $ x \mapsto x/c.$ By our assumption, the measures $\widetilde m_c$ are compactly supported in the cube $[0,1]^n$ with $\widetilde m_c(\mathbb R^n) =M_c$. By \eqref{Sonine_hypergeom},
\begin{equation}\label{hypergeom_2}  F_{BC}\bigl(\frac{\lambda}{c} + \rho(k^\prime), k^\prime; {\boldsymbol c}\bigr) = 
\int_{\mathbb R^n} F_{BC}\bigl( \frac{\lambda}{c}+ \rho(k), k; c \xi\bigr)\, d\widetilde m_c(\xi)\end{equation}
for all $\lambda\in \mathbb C^n.$ 
We may now apply Prohorov's theorem (see e.g. \cite{Bi}) and conclude 
 that the set $\{\widetilde m_c\,, 0<c < c_0\},$  is relatively sequentially compact in the weak topology.  
Hence
there exist a sequence $c_j \to 0$ and a positive Radon measure   $m$ supported on $[0,1]^n$ such 
that $\widetilde m_{c_j}\to m$ weakly for $j \to \infty$. 
Put $k_0:=(k_1+k_2, k_3)$ and $k^\prime_0:= (k_1+k_2 +h_1+h_2, k_3)$ on root system $B_n$. Using Lemma \ref{scaling} and taking the limit $c_j \to 0$ 
in formula \eqref{hypergeom_2}, we obtain that 
$$ J_{k_0^\prime}^B(\lambda, {\bf 1}) = \int_{\mathbb R^n} J_{k_0}^B(\lambda, \xi) dm(\xi)$$ 
for all $\lambda \in \mathbb C^n$. Corollary \ref{main_Bessel} now implies the assertion.
\end{proof}

 We now turn to the Heckman-Opdam polynomials of type $BC_n$. These
  have been extensively studied in the literature, see for instance \cite{BO, La, RR}.  We consider the ($W(B_n)$-invariant) normalized polynomials 
  $R_\lambda= R_\lambda^{BC_n}, \, \lambda\in \mathbb Z_+^n$.
  A short calculation shows that the rescaled polynomials
$ \widetilde R_\lambda$ on  $[0,1]^n$ defined by $$ \widetilde R_\lambda\bigl(\frac{1}{2}(1-\cos t)\bigr):= R_\lambda(k;t),\quad 
\lambda \in \mathbb Z_+^n$$
 form an  orthogonal basis of
$L^2([0,1]^n, \rho_k)$ with the weight function 
$$ \rho_k(x) = \prod_{i=1}^n x_i^{k_1+k_2-1/2}(1-x_i)^{k_2-1/2} \prod_{i<j}|x_i-x_j|^{2k_3}.$$
 
In the rank one case, the Heckman-Opdam polynomials can be written in terms of the 
 classical one-variable Jacobi polynomials
 $$ R_n^{(\alpha, \beta)}(x) = \, _2F_1\bigl(-n, n+\alpha + \beta+1; \alpha +1;\frac{1}{2}(1-x)\bigr) \quad 
 (\alpha, \beta >-1, \, n\in \mathbb Z_+)$$
as follows, see \cite[Ex.1.3.2]{HS}:
 $$ R_n^{BC_1}(k;t) = R_n^{(\alpha, \beta)}(\cos t) \,\,  \text{ with } \,\alpha = k_1+k_2 -\frac{1}{2}, \, \beta = k_2-\frac{1}{2}.$$
Classical Jacobi polynomials have various interesting integral representations. Among them are the following ones (for $x\in [-1,1]$ and with suitable probability measures $\mu_x$ depending on the parameters):
for $\gamma > \alpha > -1,$ 
\begin{equation}\label{Ultra}  
R_n^{(\gamma, \gamma)}(x) = \int_{-1}^1 R_n^{(\alpha,\alpha)}(y)d\mu_x(y) \quad \forall n\in \mathbb Z_+
\end{equation}
and for $\gamma > \alpha > -1$ and $\beta >-1$, 
\begin{equation}\label{As} R_n^{(\gamma, \beta)}(x) = \int_{-1}^1 R_n^{(\alpha,\beta)}(y)d\mu_x(y) \quad \forall n\in \mathbb Z_+\,.
\end{equation} 
see \cite{A2}, Theorem 3.4 and \cite[eq. (4.19)]{A1}.

In the higher rank case, one may ask for integral representations between Heckman-Opdam polynomials with different multiplicities.  
Note that for the normalized Heckman-Opdam polynomials $R_\lambda$  of type $BC_n$, Lemma \eqref{scaling}  implies that
\begin{equation}\label{polylimes} \lim_{m\to \infty} R_{m\lambda}\bigl(k;\frac{t}{m}\bigr) = J_{k_0}^B(\lambda, it).\end{equation}

 The following necessary condition concerns  analogues of formulas \eqref{Ultra} and   \eqref{As} in higher rank; it is a counterpart of 
 Theorem \ref{hypergeom_integral} with essentially the same proof.

\begin{proposition} Consider root system $BC_n$ with multiplicities $k=(k_1,k_2,k_3)$ and $ k^\prime = (k_1+h_1,k_2+h_2, k_3) $  
as in Proposition \ref{hypergeom_integral}.  Suppose that for each $\tau \in \mathbb R/2\pi \mathbb Z\,$ 
there exists a positive Radon measure $m_\tau$ on $\mathbb T= (\mathbb R/2\pi \mathbb Z)^n$ such that 
$$ R_\lambda(k^\prime; {\bf \tau}) = \int_{\mathbb T} R_\lambda(k;s) dm_\tau(s) \quad \forall \lambda \in \mathbb Z_+^n.$$ 
Then $h_1+h_2$ is contained in $\Sigma(k_3).$

\end{proposition}

We finally turn to branching rules for Heckman-Opdam polynomials of type $BC_n$. For multiplicities $k, k^\prime$ on $BC_n$ and 
$\lambda\in \mathbb Z_+^n$ we have an expansion
$$ R_\lambda(k^\prime,t) = \sum_{\nu \leq \lambda } c_{\lambda,\nu}(k^\prime,k) R_\nu(k;t)$$
with unique connection coefficients $c_{\lambda,\nu}(k^\prime,k)\in \mathbb R.$ In rank one, the following positivity result for 
the connection coefficients between Jacobi polynomial systems is well-known, see e.g. \cite[(7.33)]{A2}: For $\alpha, \beta > -1$ and $\nu >0,$ 
 $$ R_n^{(\alpha + \nu, \beta)} = \sum_{j=0}^n c_{n,j} R_j^{(\alpha, \beta)} \,\, \text{ with }\, c_{n,j}\geq 0.$$ 
Heckman-Opdam polynomials of type $BC_n$ generalize the spherical functions of compact Grassmannians. See \cite{RR} for a detailed treatment, 
where however the notation (scaling of root systems and multiplicities) is slightly different from ours.  
To become specific, consider for fixed $n\in \mathbb N$ and integers $m>n$ the compact Grassmann manifolds $U_m/K_m$ with $U_m=SU(m+n, \mathbb F), 
K_m=S(U(m,\mathbb F) \times U(n,\mathbb F))$ for  $\mathbb F\in \{\mathbb R, \mathbb C, \mathbb H\}.$ Via polar decomposition of $U_m$, the 
double coset space $U_m//K_m$ may be topologically identified with the fundamental alcove
$$ A_0 = \{t \in \mathbb R^n: \frac{\pi}{2} \geq t_1 \geq \ldots \geq t_n\geq 0\}$$
with $t \in A_0$ being identified with the matrix
$$ a_t = \begin{pmatrix} \cos \underline t & \,0  & -\sin\underline t\\
0 \, & \,I_{m-n} & \, 0 \\
\sin\underline t & \, 0  & \cos\underline t \end{pmatrix} \in U_m\,, \quad \underline t = \text{diag}(t_1, \ldots, t_n).$$ 
The spherical functions of $U_m/K_m$ are given by
$$ \varphi_\lambda^m(a_t) = R_\lambda(k_m;2t),\quad \lambda\in \mathbb Z_+^n$$
with
$$ k_m = \bigl(d(m-n)/2, (d-1)/2, d/2\bigr), \quad d= \text{dim}_{\mathbb R}\mathbb F. $$
Here the $R_\lambda$ are again the normalized Heckman-Opdam polynomials associated with root system $BC_n$. 
For integers $l>m$ we consider $U_m$ as a closed subgroup of $U_l$. Then  
$K_m = U_m\cap K_l$ and $U_m/K_m$ is a submanifold of $U_l/K_l$.  As a function on $U_l$, the spherical function $\varphi_\lambda^l$ is 
$K_l$-biinvariant and positive definite, and its restriction to $U_m$ is $K_m$-biinvariant and positive definite on $U_m$. This implies that
$$ \varphi_\lambda^l\vert_{U_m} = \sum_{\nu \in \mathbb Z_+^n} c_{\lambda, \nu} \varphi_\nu^m $$
with unique branching coefficients $c_{\lambda, \nu} = c_{\lambda, \nu}(l,m)\geq 0$, only finitely many of them being different from zero. 
For the Heckman-Opdam polyomials this implies that
$$ R_\lambda(k_l;t) = \sum_{\nu \leq \lambda} c_{\lambda,\nu} R_{\nu}(k_m;t),$$
so the connection coefficients between the two systems are non-negative.

The next result however shows that for general multiplicities $k = (k_1, k_2, k_3) \geq 0$ and   $k^\prime = (k_1+h, k_2, k_3)$ with $h>0$   
 there may also occur negative connection coefficients between the associated systems of Heckman-Opdam polynomials.

 \begin{theorem}\label{connection_HO} Consider root system $BC_n$ with multiplicities $k=(k_1,k_2,k_3)$ and 
 $ k^\prime = (k_1+h_1,k_2+h_2, k_3) $  as in Theorem \ref{hypergeom_integral}. Suppose that $h_1+h_2 \notin \Sigma(k_3).$ 
 Then the connection coefficients in the expansion
\begin{equation}\label{expansion} R_{\bf m}(k^\prime;t)  = \sum_{\nu \leq {\bf m}} c_{{\bf m}, \nu} R_{\nu}(k;t), 
\quad {\bf m} = (m, \ldots, m)\end{equation}
 satisfy
  $$ \sup_{m\in \mathbb N} \sum_{\nu \leq \,{\bf m}} | c_{{\bf m}, \nu}| = \infty.$$
 In particular, there exist infinitely many $m\in \mathbb N$ such that $c_{{\bf m}, \nu} < 0$ for some $\nu$. 
 \end{theorem}
 
 \begin{proof} Assume in the contrary that $\,S:= \sup_{m\in \mathbb N} \sum_{\nu \leq \,{\bf m}} | c_{{\bf m}, \nu}| < \infty.$ 
 We proceed similar as in \cite{RV1} and introduce the bounded, discrete signed measures
 $$  \mu_m:= \sum_{\nu \leq {\bf m}}  c_{{\bf m}, \nu}\,\delta_{\nu/m} \in M_b(\mathbb R^n), \quad m \in \mathbb N,$$
 where $\delta_x$ denotes the point measure in $x\in \mathbb R^n.$ 
By definition of the dominance order, the support of $\mu_m$  is contained in the compact cube $[0,1]^n.$
 With these measures,  
 expansion \eqref{expansion} can be written as
  \[
 R_{\bf m}\bigl( k^\prime; \frac{t}{m}\bigr)  =  \sum_{\nu \leq {\bf m}}  c_{{\bf m}, \nu}\, F_{BC} \bigl(\nu + \rho(k),k;\frac{it}{m}\bigr) 
  = \int_{[0,1]^n} F_{BC}\bigl(mx + \rho(k),k;\frac{it}{m}\bigr)d\mu_m(x).
 \]

 We consider the Jordan decomposition $\mu_m = \mu_m^1 -\mu_m^2 $ where $\mu_m^i$ are positive measures whose total variation norm 
 satisfies $\|\mu_m^i\| \leq \|\mu_m\| \leq S$ and which are supported in $[0,1]^n.$ Using again Prohorov's theorem we obtain, after 
 passing to subsequences if necessary, that there exist positive bounded Radon measures $\mu^i$ on $\mathbb R^n$ with 
 $\text{supp}(\mu^i) \subseteq [0,1]^n$ and such that $\mu_m^i \to \mu ^i$ (weakly) as $m \to \infty.$ Therefore $\,\mu_m \to \mu:= \mu^1 - \mu^2$.
 Taking the limit $m\to \infty$ and employing  Lemma \ref{scaling} as well as formula \eqref{polylimes}, we obtain
 $$ J_{k_0^\prime}^B ({\bf 1}, it)  =    \int_{[0,1]^n} J_{k_0}^B(\xi,it) d\mu(\xi)   \quad \forall t \in \mathbb R^n,    $$
 with $k_0 = (k_1+k_2, k_3), k_0^\prime = (k_1+k_2+h_1+h_2, k_3).$  
 Again, Corollary \ref{main_Bessel} now implies that $h_1+h_2 \in \Sigma(k_3)$, a contradition.  
 \end{proof}

\begin{remark} Apart from the well-studied rank one case (see \cite{A2} for an overview) and the geometric cases described above, 
further nontrivial pairs of Heckman-Opdam polynomial families with nonnegative connection coefficients seem to be unknown. 
\end{remark}

\end{document}